\documentclass[12pt]{amsart}
\usepackage{geometry}
\usepackage[mathscr]{euscript}
\usepackage{amssymb}
\usepackage[comma, numbers]{natbib}
\usepackage{bm}
\usepackage{enumerate}
\usepackage[english]{babel}
\usepackage{commath}
\usepackage{graphicx}

\theoremstyle{plain}
\newtheorem{theorem}{Theorem}[section]

\theoremstyle{plain}
\newtheorem{lemma}[theorem]{Lemma}

\theoremstyle{plain}
\newtheorem{corollary}[theorem]{Corollary}

\theoremstyle{definition}
\newtheorem{definition}[theorem]{Definition}

\theoremstyle{plain}

\theoremstyle{remark}

\theoremstyle{definition}
\newtheorem{example}[theorem]{Example}

\theoremstyle{plain}

\theoremstyle{plain}

\theoremstyle{plain}

\theoremstyle{plain}

\title{Inductive Dimensions of Coarse Proximity Spaces}

\author{Pawel Grzegrzolka}
\address{Syracuse University, Syracuse, USA}
\email{pgrzegrz@syr.edu}

\author{Jeremy Siegert}
\address{University of Tennessee, Knoxville, USA} 
\email{jsiegert@vols.utk.edu}
\date{\today} 


\begin{document}

\begin{abstract}
In this paper, we generalize Dranishnikov's asymptotic inductive dimension to the setting of coarse proximity spaces. We show that in this more general context, the asymptotic inductive dimension of a coarse proximity space is bigger or equal to the inductive dimension of its boundary, and consequently may be strictly bigger than the covering dimension of the boundary. We also give a condition, called complete traceability, on the boundary of the coarse proximity space under which the asymptotic inductive dimension of a coarse proximity space and the inductive dimension of its boundary coincide. Finally, we show that spaces whose boundaries are $Z$-sets and spaces admitting metrizable compactifications have completely traceable boundaries.
\end{abstract}

\maketitle
\tableofcontents

\section{Introduction}

The coarse dimensional invariant called asymptotic dimension was originally described by Gromov in \cite{Gromov} and was used to study infinite discrete groups. The theory was later extended more broadly to proper metric spaces, with particular interest being generated by Yu's result relating the property of a proper metric space having finite asymptotic dimension to the Novikov conjecture (see \cite{Yu}). 

Asymptotic dimension provides a coarse analog of Lebesgue covering dimension via a ``going to infinity" perspective. It is defined by specifying what happens at particular ``scales" (represented by uniformly bounded families) within a metric space. This stands in contrast to an alternative perspective in coarse geometry which focuses on properties defined ``at infinity," typically by means of defining properties on boundary spaces associated to metric spaces such as the Higson corona of proper metric spaces or the Gromov boundary of hyperbolic metric spaces. The prototypical definition of a large-scale dimension of a proper metric space ``at infinity" is the covering dimension of the Higson corona of the space. A strong relationship between asymptotic dimension of a proper metric space and the covering dimension of its Higson corona was found by Dranishnikov who proved that if a proper metric space has finite asymptotic dimension, then its asymptotic dimension and the covering dimension of the Higson corona agree (see \cite{Dranishnikovasymptotictopology}). Pursuing the dimension theory of Higson coronae in more detail, Dranishnikov went on to define the asymptotic inductive dimension 
of proper metric spaces in \cite{asymptoticinductivedimension}. 
This dimensional invariant is meant to serve as a coarse analog of the large inductive dimension of topological spaces studied in classical dimension theory. As the large inductive dimension and the covering dimension coincide in the class of metrizable spaces, Dranishnikov asked if the asymptotic inductive dimension of a proper metric space coincides with the covering dimension of its Higson corona. Alongside this question, the problem of wether or not the asymptotic inductive dimension of a proper metric space coincides with the large inductive dimension of its Higson corona was posed. 

In \cite{paper1}, coarse proximity spaces were introduced to axiomatize the ``at infinity" perspective of coarse geometry, providing general definitions of coarse neighborhoods (whose metric space specific definition was given by Dranishnikov in \cite{Dranishnikovasymptotictopology}), asymptotic disjointness, and closeness ``at infinity." As seen in \cite{paper3} some coarse invariants, such as the Higson corona are intuitively described using the language of coarse proximities. Coarse proximity structures lie between metric spaces and coarse spaces (as defined by Roe in \cite{Roe}) in a way similar to how proximity spaces relate to metric spaces and uniform spaces (see \cite{paper2} or \cite{isbell}). In the same way that small scale proximity structures provide an intuitive description of all possible compactifications of locally compact Hausdorff spaces, coarse proximity structures on similar spaces provide intuitive descriptions of the boundaries of such compactifications. In particular, in \cite{paper3}, the authors construct a functor from the category of coarse proximity spaces to the category of compact Hausdorff spaces that assigns to each coarse proximity space a certain ``boundary space." This functor provides a common language for speaking of boundary spaces such as the Higson corona, the Gromov boundary, and other well-known boundary spaces. In this paper, we generalize the notion of asymptotic inductive dimension to all coarse proximity spaces (whose definition agrees with Dranishnikov's definition for proper metric spaces) and investigate both of Dranishnikov's questions in this more general context. In section \ref{preliminaries}, we review the necessary background information surrounding proximities as well as coarse proximities and their boundaries. In section \ref{dimensionofcoarseproximityspaces}, we define the asymptotic inductive dimension of coarse proximity spaces and show that it is an invariant within the category of coarse proximity spaces. We also show that the answer to Dranishnikov's first question (``Does the asymptotic inductive dimension of the proper metric space coincide with the covering dimension of its Higson corona?") generalized to this broader context is negative. In section \ref{maintheorems}, we describe two classes of completely traceable coarse proximity spaces in which the answer to the second of Dranishnokov's questions (``Does the asymptotic inductive dimension of a proper metric space coincide with the large inductive dimension of its Higson corona?") generalized to this broader context is positive. Specifically, these are locally compact Hausdorff spaces that admit metrizable compactification and spaces admitting compactifications whose boundaries are $Z$-sets. These classes include well-known boundaries such as the Gromov and visual boundaries, as well as the boundaries of the ``coarse-compactification," described in \cite{yamauchi}. Finally, in section \ref{Freud}, we apply the results of the paper to the Freudenthal boundary of a certain class of metrizable spaces.

\section{Preliminaries}\label{preliminaries}
In this section, we provide the basic definitions and theorems surrounding proximity spaces and coarse proximity spaces. The definitions and theorems about proximity spaces come from \cite{proximityspaces}, and the definitions and theorems about coarse proximity spaces come from \cite{paper3}.

\begin{definition}
Let $X$ be a set. A binary relation $\delta$ on the power set of $X$ is called a {\bf proximity} on $X$ if it satisfies the following axioms for all $A,B,C \subseteq X$:
\begin{enumerate}
	\item $A\delta B\implies B\delta A,$
	\item $A\delta B\implies A,B\neq\emptyset,$
	\item $A \cap B\neq\emptyset\implies A\delta B,$
	\item $A\delta (B\cup C)\iff A\delta B\text{ or }A\delta C,$
	\item $A\bar{\delta}B\implies\exists E\subseteq X,\,A\bar{\delta}E\text{ and }(X\setminus E)\bar{\delta}B,$
\end{enumerate}
where by $A\bar{\delta}B$ we mean that the negation of $A\delta B$ holds. If in addition to these axioms a proximity satisfies $\{x\}\delta\{y\}\iff x=y$ for all $x,y\in X,$ we say that the proximity $\delta$ is {\bf separated}. A pair $(X,\delta)$ where $X$ is a set and $\delta$ is a proximity on $X$ is called a {\bf proximity space}.
\end{definition}

\begin{definition}
Let $(X, \delta_1)$ and $(Y, \delta_2)$ be two proximity spaces. A function $f:X \to Y$ is called a \textbf{proximity map} if
\[A \delta_1 B \implies f(A) \delta_2 f(B).\]
A bijective proximity map whose inverse is also a proximity map is called a \textbf{proximity isomorphism}. A proximity isomorphism onto a subspace of a proximity space is called a \textbf{proximity embedding}.
\end{definition}

The topology of a proximity space $(X,\delta)$ is defined by means of the closure operator defined by
\[cl(A)=\{x\in X\mid \{x\}\delta A\}.\]
This topology is referred to as the {\bf induced topology} of $\delta$. It is always completely regular, and is Hausdorff if and only if the proximity $\delta$ is separated.

\begin{example}
Let $X$ be a Tikhonov space. That is, $X$ is a Hausdorff space and for every point $x\in X$ and nonempty closed subset $A\subseteq X$ not containing $x$, there is a continuous function $f:X\rightarrow [0,1]$ such that $f(x)=0$ and $f(A)=1$. One can define a proximity relation $\delta$ by defining $A\delta B$ for subsets $A$ and $B$ of $X$ if and only if there is no function $f:X\rightarrow [0,1]$ such that $f(A)=0$ and $f(B)=1$. The topology induced by this proximity is the original topology on $X$. 
\end{example}

\begin{example}
If $(X,d)$ is a metric space then the relation $\delta$ defined by $A\delta B$ if and only if $d(A,B)=0$ is a proximity relation that induces the metric topology.

\end{example}

\begin{example}\label{discrete proximity}
For any nonempty set $X$ the relation $\delta$ defined on subsets of $X$ by $A\delta B$ if and only if $A\cap B\neq\emptyset$ is a proximity relation that induces the discrete topology on $X$. We call this proximity the {\bf discrete proximity}. 

\end{example}

Every separated proximity space admits a unique (up to a proximity isomorphism) compactification, which we describe briefly below.

\begin{definition}
	A {\bf cluster} in a separated proximity space $(X,\delta)$ is a collection $\sigma$ of nonempty subsets of $X$ satisfying the following:
	\begin{enumerate}
		\item For all $A,B\in\sigma$, $A\delta B,$
		\item If $C\delta A$ for all $A\in\sigma,$ then $C\in\sigma,$
		\item If $(A\cup B)\in\sigma,$ then $A\in\sigma$ or $B\in\sigma.$
	\end{enumerate}
A cluster $\sigma$ is called a {\bf point cluster} if $\{x\}\in\sigma$ for some $x\in X.$
\end{definition}
One property of clusters that we are going to use without mention is that given two clusters $\sigma_1$ and $\sigma_2$, one has $\sigma_1\ \subseteq \sigma_2 \implies \sigma_1=\sigma_2$.

The set of all clusters in a separated proximity space is denoted by $\mathfrak{X}$. Given a set $\mathcal{A}\subseteq\mathfrak{X}$ and a subset $C\subseteq X,$ we say that $C$ {\bf absorbs} $\mathcal{A}$ if $C\in\sigma$ for every $\sigma\in\mathcal{A}$.

\begin{theorem}\label{theorem111}
	Let $(X,\delta)$ be a separated proximity space and $\mathfrak{X}$ the corresponding set of clusters. The relation $\delta^{*}$ on the power set of $\mathfrak{X}$ defined by
	\[\mathcal{A}\delta^{*}\mathcal{B}\iff A\delta B\]
	for all sets $A,B\subseteq X$ that absorb $\mathcal{A}$ and $\mathcal{B}$, respectively, is a proximity on $\mathfrak{X}.$ In fact, $(\mathfrak{X}, \delta^*)$ is a compact separated proximity space into which $X$ proximally embeds as a dense subspace (by mapping each point to its corresponding point cluster). \hfill $\square$
\end{theorem}

The compactification $(\mathfrak{X}, \delta^*)$ described above is the {\bf Smirnov compactification} of the proximity space $(X,\delta)$. We call the subset $\mathfrak{X}\setminus X$ the {\bf Smirnov boundary} of $X$. It is customary to identify $X$ with its image in $\mathfrak{X}.$ Consequently, given $A \subseteq X,$ one can think of $A$ as a subset of $\mathfrak{X}.$ It is then easy to show that
\[cl_{\mathfrak{X}}(A) = \{ \sigma \in \mathfrak{X} \mid A \in \sigma\}.\]
 In other words, 
\[A \in \sigma \iff \sigma \in cl_{\mathfrak{X}}(A).\]

Now we recall basic definitions and theorems surrounding coarse proximity spaces, as found in \cite{paper3}.

\begin{definition}
	A {\bf bornology} $\mathcal{B}$ on a set $X$ is a family of subsets of $X$ satisfying:
	\begin{enumerate}
		\item $\{x\}\in\mathcal{B}$ for all $x\in X,$
		\item $A\in\mathcal{B}$ and $B\subseteq A$ implies $B\in\mathcal{B},$
		\item If $A,B\in\mathcal{B},$ then $A\cup B\in\mathcal{B}.$
	\end{enumerate}
	Elements of $\mathcal{B}$ are called {\bf bounded} and subsets of $X$ not in $\mathcal{B}$ are called {\bf unbounded}.
\end{definition}

\begin{definition}\label{coarseproximitydefinition}
	Let $X$ be a set equipped with a bornology $\mathcal{B}$. Let $A,B,$ and $C$ be subsets of $X$. A \textbf{coarse proximity} on a set $X$ is a relation ${\bf b}$ on the power set of $X$ satisfying the following axioms:
	
	\begin{enumerate}
		\item $A{\bf b}B \implies B{\bf b}A,$ \label{axiom1}
		\item $A{\bf b}B \implies A,B  \notin \mathcal{B},$ \label{axiom2}
		\item $A\cap B \notin \mathcal{B} \implies A {\bf b} B,$ \label{axiom3}
		\item $(A \cup B){\bf b}C \iff A{\bf b}C$ or $B{\bf b}C,$ \label{axiom4}
		\item $A\bar{\bf b}B\implies\exists E\subseteq X, A\bar{\bf b}E$ and $(X\setminus E)\bar{\bf b}B,$ \label{axiom5}
	\end{enumerate}
where by $A\bar{\bf b}B$ we mean that the negation of $A{\bf b} B$ holds. If $A {\bf b} B$, then we say that $A$ is \textbf{coarsely close} to 
 $B.$ 
A triple $(X,\mathcal{B},{\bf b})$ where $X$ is a set, $\mathcal{B}$ is a bornology on $X$, and ${\bf b}$ is a coarse proximity relation on $X,$ is called a {\bf coarse proximity space}.
\end{definition}

We note here that the definition of a coarse proximity resembles the definition of a proximity, similar to the way the definition of a coarse space resembles the definition of a uniform space (see \cite{Roe} for more on coarse spaces). The substantial change in the definition of a coarse proximity is the replacement of an empty set with any bounded set. This change is what is needed to capture ``closeness at infinity" (see \citep{paper1}, \cite{paper3}), as opposed to capturing ``closeness within the space" (see \cite{proximityspaces}). It is also worth noting that a coarse proximity structure is built purely around the notion of ``closeness at infinity," without utilizing the prior notion of scale (as in the case of coarse spaces). All that is needed is the notion of boundness. For more on the relationship between coarse proximity spaces and coarse spaces, see \cite{paper2}.

\begin{example}\label{subspace definition}
	Let $(X,\mathcal{B},{\bf b})$ be a coarse proximity space and $Y\subseteq X$ any subset. Then the coarse proximity structure on $Y$ given by the bornology
	\[\mathcal{B}_{Y}=\{B\cap Y\mid B\in\mathcal{B}\}\]
	and the binary relation ${\bf b}_{Y}$ defined by
	\[A{\bf b}_{Y}C \iff A{\bf b}C \text{ (as subsets of $X$)}\]
	makes $(Y, \mathcal{B}_{Y}, {\bf b}_Y)$ into a coarse proximity space. This coarse proximity structure is called the \textbf{subspace coarse proximity structure} on $Y$.
\end{example}

The proof that the following example is a coarse proximity space can be found in \cite{paper1}.
\begin{example}\label{metric coarse proximity structure}
	Let $(X,d)$ be a metric space and $\mathcal{B}$ the set of all metrically bounded sets in $X$. The relation ${\bf b}$ defined by
	\[A{\bf b}B\iff\exists\epsilon>0,\,\forall D\in\mathcal{B},\,d(A\setminus D,B\setminus D)<\epsilon\]
 	makes $(X,\mathcal{B},{\bf b})$ into a coarse proximity space. This coarse proximity structure is called the {\bf metric coarse proximity structure} on $(X,d)$.
\end{example}

The proof that the following example is a coarse proximity space can be found in \cite{paper3}.
\begin{example}\label{coarse proximity structure induced by a compactification}
	Let $X$ be a locally compact Hausdorff space and $\overline{X}$ a compactification thereof. Define $\mathcal{B}$ to be the set of all $K\subseteq X$ for which $cl_{X}(K)$ is compact. Then the relation ${\bf b}$ defined by
	\[A{\bf b}B\iff cl_{\overline{X}}(A)\cap cl_{\overline{X}}(B)\cap(\overline{X}\setminus X)\neq\emptyset\]
	makes $(X,\mathcal{B},{\bf b})$ into a coarse proximity space. This coarse proximity structure is called the {\bf coarse proximity structure induced by the compactification} $\overline{X}$. 
\end{example}

\begin{definition}
	Let $(X,\mathcal{B}_1,{\bf b}_1)$ and $(Y,\mathcal{B}_2,{\bf b}_2)$ be coarse proximity spaces. Let $f:X\rightarrow Y$ be a function. Then $f$ is a \textbf{coarse proximity map} provided that the following are satisfied for all $A,B \subseteq X$:
\begin{enumerate}
\item $B\in\mathcal{B}_1 \implies f(B)\in\mathcal{B}_2,$ \label{item11}
\item $A{\bf b}_1B \implies f(A){\bf b}_2f(B).$ \label{item22}
\end{enumerate}
\end{definition}

For the proof of the following theorem, see \cite{paper1}.

\begin{theorem}
	Let $(X,\mathcal{B},{\bf b})$ be a coarse proximity space. Let $\phi$ be the relation on the power set of $X$ defined in the following way: $A\phi B$ if and only if the following hold:
	\begin{enumerate}
	\item for every unbounded $B^{\prime}\subseteq B$ we have $A{\bf b}B^{\prime},$
	\item for every unbounded $A^{\prime}\subseteq A$ we have $A^{\prime}{\bf b}B.$
	\end{enumerate}	
	Then  $\phi$ is an equivalence relation satisfying 
	\[A\phi B \text{ and } C \phi D \implies (A\cup C) \phi (B \cup D)\]
	for any $A,B,C,D \subseteq X.$ We call this equivalence relation the \textbf{weak asymptotic resemblance} induced by the coarse proximity ${\bf b}$. \hfill $\square$
\end{theorem}

When $X$ is a metric space and ${\bf b}$ is the metric coarse proximity structure, then $\phi$ is equivalent to the relation of having finite Hausdorff distance (see \cite{paper1}).

\begin{definition}\label{coarsecloseness}
	Let $X$ be a set and $(Y,\mathcal{B},{\bf b})$ a coarse proximity space. Two functions $f,g:X \to Y$ are {\bf close} if for all $A \subseteq X$
	\[f(A)\phi g(A),\]
	where $\phi$ is the weak asymptotic resemblance relation induced by the coarse proximity structure ${\bf b}.$
\end{definition}

\begin{definition}
Let $(X,\mathcal{B}_1,{\bf b}_1)$ and $(Y,\mathcal{B}_2,{\bf b}_2)$ be coarse proximity spaces. We call a coarse proximity map $f: X \to Y$ a \textbf{coarse proximity isomorphism} if there exists a coarse proximity map $g:Y\to X$ such that $g\circ f$ is close to $id_{X}$ and $f\circ g$ is close to $id_{Y}.$ We say that $(X,\mathcal{B}_1,{\bf b}_1)$ and $(Y,\mathcal{B}_2,{\bf b}_2)$ are \textbf{coarse proximity isomorpic} if there exists a coarse proximity isomorphism $f:X \to Y.$
\end{definition}

The collection of coarse proximity spaces and closeness classes of coarse proximity maps makes up the category of coarse proximity spaces. Also, a function $f$ between two metric spaces is a coarse map if and only if it is a coarse proximity map between the induced metric coarse proximity spaces. For details, see \cite{paper1}.

As it was shown in \cite{paper3}, each coarse proximity space $(X,\mathcal{B},{\bf b})$ has the associated \textbf{boundary} $\mathcal{U}X$. This boundary space $\mathcal{U}X$ is constructed as follows: let $\delta$ to be a proximity on $X$ defined by
\[A\delta B\iff (A\cap B\neq\emptyset\text{ or }A{\bf b}B).\]
This $\delta$ is a separated proximity called the {\bf discrete extension} of the coarse proximity ${\bf b}$ (one might compare this with the discrete small scale proximity in Example \ref{discrete proximity}). Then $\mathcal{U}X$ is defined to be a subset of \ $\mathfrak{X}$ (where $\mathfrak{X}$ denotes the Smirnov compactification associated to $\delta$) containing only those clusters that do not contain any bounded sets. It was shown in \cite{paper3} that $\mathcal{U}X$ is compact and Hausdorff, and that $\mathcal{U}X$ encodes the asymptotic behavior of subsets of $X.$ In particular, if for $A\subseteq X$, one defines the {\bf trace} of $A$, denoted by $A'$, by
\[A'=cl_{\mathfrak{X}}(A)\cap\mathcal{U}X,\]
then for $A,B\subseteq X,$ one has
\begin{enumerate}
\item $A{\bf b}B \iff A'\cap B' \neq \emptyset,$
\item $A \phi B \iff A' = B',$
\item $A,B \in \sigma\in\mathcal{U}X$ $\implies$ $A{\bf b}B.$
\end{enumerate}

Every coarse proximity map between coarse proximity spaces induces a map between the boundaries of these coarse proximity spaces, as the following definition shows.
\begin{definition}\label{mapUf}
Let $f:(X,\mathcal{B}_1,{\bf b}_1)\rightarrow(Y,\mathcal{B}_2,{\bf b}_2)$ be a coarse proximity map. Then the map $\mathcal{U}f: \mathcal{U}X \to \mathcal{U}Y$ is defined by
\[\mathcal{U}f(\sigma)=\{D\subseteq Y\mid D{\bf b}_{2}f(C) \text{ for all }C\in\sigma\}.\]
\end{definition}
It was shown in \cite{paper3} that the above map is continuous and agrees on $\mathcal{U}X$ with the unique extension of $f$ to the Smirnov compactifications $\mathfrak{X}$ and $\mathfrak{Y}$ of $X$ and $Y$, respectively (where $\mathfrak{X}$ and $\mathfrak{Y}$ are Smornov compactifications of the respective discrete extensions).

We finish this section with two important examples of boundaries of coarse proximity spaces. For details on these examples, see \cite{paper3}.

\begin{example}
	If $(X,d)$ is a proper metric space (where proper means that closed bounded subsets of $X$ are compact) equipped with its metric coarse proximity structure as in Example \ref{metric coarse proximity structure}, then $\mathcal{U}X$ is homeomorphic to the Higson corona $\nu X$ of $X$.
\end{example}

\begin{example}\label{important_example}
	Let $X$ be a locally compact Hausdorff space with compactification $\overline{X}$. If $X$ is equipped with the coarse proximity structure induced by the compactification $\overline{X}$ as in Example \ref{coarse proximity structure induced by a compactification}, then $\mathcal{U}X$ is homeomorphic to $\overline{X}\setminus X$.
\end{example}

\section{Asymptotic inductive dimension of coarse proximity spaces}\label{dimensionofcoarseproximityspaces}

In this section, we define asymptotic inductive dimension of coarse proximity spaces. Our definition is a generalization of the definition of asymptotic inductive dimension for proper metric spaces as defined by Dranishnikov in \cite{asymptoticinductivedimension}, whose definition is a coarse analog of the large inductive dimension. 
Our definition will provide an invariant within the category of coarse proximity spaces.

We first review the definition of the large inductive dimension. For a thorough treatment of the theory thereof, see \cite{engelking}.

\begin{definition}
	Given disjoint subsets $A$ and $B$ of a topological space $X$, a {\bf separator} between them is a subset $C\subseteq X$ such that there are disjoint open sets $U,V\subseteq X$ such that $X\setminus C=U\cup V$, $A\subseteq U$, and $B\subseteq V$.
\end{definition}

Notice that separators are necessarily closed.

\begin{definition}\label{large inductive dimension}
	Let $X$ be a normal space. The {\bf large inductive dimension} of $X,$ denoted $Ind(X)$, is defined in the following way:
	\begin{itemize}
		\item $Ind(X)=-1$ if and only if $X$ is empty;
		\item for $n\geq 0$, $Ind(X)\leq n$ if for every pair of disjoint closed subsets $A,B\subseteq X,$ there exists a separator $C\subseteq X$ between $A$ and $B$ such that $Ind(C)\leq n-1;$
		\item $Ind(X)=n$ if $n\geq -1$ is the smallest integer for which $Ind(X)\leq n$ holds;
		\item if $Ind(X)\leq n$ doesn't hold for any integer, then we say that $Ind(X)=\infty$.
	\end{itemize}
\end{definition}

The following coarse analog of large inductive dimension for proper metric spaces  given by Dranishnikov can be found in \cite{asymptoticinductivedimension} and \cite{Bell}. To understand the definition, recall that two subsets $A$ and $B$ of a metric space are called \textbf{asymptotically disjoint} if and only if $\lim_{r \to \infty}d(A\setminus N_r(x_0), B \setminus N_r(x_0)) = \infty$ for any base point $x_0,$ where $N_r(x_0)$ denotes the ball of radius $r$ with the center $x_0.$  

\begin{definition}
	Given a proper metric space $(X,d)$ and two subsets $A, B \subseteq X$ that are asymptotically disjoint, a set $C\subseteq X$ is an \textbf{asymptotic separator} between $A$ and $B$ if the trace of $C$ in $\nu X$ (i.e., the intersection of the closure of $C$ in the Higson compactification and the Higson Corona $\nu X$) is a separator in $\nu X$ between traces of $A$ and $B$ in $\nu X.$
\end{definition}

\begin{definition}\label{asymptotic_inductive_dimension_definition}
	Let $(X,d)$ be a proper metric space. The {\bf asymptotic inductive dimension} of $X,$ denoted $asInd(X)$, is defined in the following way:
	\begin{itemize}
		\item $asInd(X)=-1$ if and only if $X$ is bounded;
		\item for $n\geq 0$, $asInd(X)\leq n$ if for every pair of asymptotically disjoint subsets $A,B\subseteq X,$ there exists an asymptotic separator $C\subseteq X$ between $A$ and $B$ such that $asInd(C)\leq n-1;$
		\item $asInd(X)=n$ if $n\geq -1 $ is the smallest integer for which $asInd(X)\leq n$ holds;
		\item if $asInd(X)\leq n$ doesn't hold for any integer, then we say that $asInd(X)=\infty$.
	\end{itemize}
\end{definition}

Now we generalize Dranishnikov's asymptotic inductive dimension to all coarse proximity spaces. 

\begin{definition}
	Given subsets $A$ and $B$ of a coarse proximity space $(X,\mathcal{B},{\bf b})$ such that $A\bar{\bf b}B$, a set $C\subseteq X$ is an \textbf{asymptotic separator} between $A$ and $B$ if $C^{\prime}$ is a separator in $\mathcal{U}X$ between $A^{\prime}$ and $B^{\prime}$.
\end{definition}

	Note that the above definition coincides with Dranishnikov's definition of an asymptotic separator in the case of proper metric spaces.
	
	\begin{definition}
	Let $(X,\mathcal{B},{\bf b})$ be a coarse proximity space. The {\bf asymptotic inductive dimension} of $X,$ denoted $asInd(X)$, is defined in the following way:
	\begin{itemize}
		\item $asInd(X)=-1$ if and only if $X$ is bounded;
		\item for $n\geq 0$, $asInd(X)\leq n$ if for $A,B\subseteq X$ such that $A\bar{\bf b}B,$ there exists an asymptotic separator $C\subseteq X$ between $A$ and $B$ such that $asInd(C)\leq n-1$ (where $C$ is equipped with the subspace coarse proximity structure);
		\item $asInd(X)=n$ if $n\geq -1 $ is the smallest integer for which $asInd(X)\leq n$ holds;
		\item if $asInd(X)\leq n$ doesn't hold for any integer, then we say that $asInd(X)=\infty$.
	\end{itemize}
\end{definition}

	Note that the above definition coincides with Dranishnikov's definition of an asymptotic inductive dimension when the proper metric space is given the metric coarse proximity structure.
	
	To show that the asymptotic inductive dimension is invariant in the category of coarse proximity spaces, we need the following two lemmas.

\begin{lemma}\label{coarse proximity is completely preserved by isomorphisms}
	Let $f:(X, \mathcal{B}_1, {\bf b}_1)\rightarrow (Y, \mathcal{B}_2, {\bf b}_2)$ be a coarse proximity isomorphism with coarse proximity inverse $g$.  Let $\phi_1$ and $\phi_2$ be weak asymptotic resemblances associated to ${\bf b}_1$ and ${\bf b}_2$, respectively. Then given $A,B\subseteq X,$ we have that:
	\begin{enumerate}
	\item $A{\bf b}_1B$ if and only if $f(A){\bf b}_2f(B),$
	\item $A\phi_{1} B$ if and only if $f(A)\phi_{2}f(B).$
	\end{enumerate}
\end{lemma}
\begin{proof}
That $A{\bf b}_1B$ implies $f(A){\bf b}_2f(B)$ is given simply by the definition of a coarse proximity map. Conversely, if $f(A){\bf b}_2f(B)$, then $gf(A){\bf b}_1gf(B)$. By the definition of coarse proximity isomorphisms, we have that $gf(A)\phi_1 A$ and $gf(B)\phi_1 B$. By Corollary $6.18$ in \cite{paper1}, we get that $A{\bf b}_1B$.

If $A\phi_{1} B,$ then $f(A)\phi_{2} f(B)$ because coarse proximity isomorphisms preserve $\phi$ relations (see Proposition 7.14 in \cite{paper1}). To see the opposite direction, let $f(A)\phi_{2}f(B).$  Then, $gf(A)\phi_{1}gf(B).$ By the definition of coarse proximity isomorphisms, this shows that $A\phi_{1}gf(A)\phi_{1}gf(B)\phi_{1}B$, which implies that $A\phi_{1}B.$ 
\end{proof}

\begin{lemma}\label{sending traces to traces}
	Let $f:(X,\mathcal{B}_{1},{\bf b}_{1})\rightarrow (Y,\mathcal{B}_{2},{\bf b}_{2})$ be a coarse proximity map and $\mathcal{U}f:\mathcal{U}X\rightarrow\mathcal{U}Y$ the corresponding continuous map between boundaries. Then for an unbounded set $A\subseteq X$ with trace $A^{\prime},$ we have that $\mathcal{U}f(A^{\prime})\subseteq f(A)^{\prime}$. If $f$ is a coarse proximity isomorphism with coarse inverse $g$, then $\mathcal{U}f(A^{\prime})= f(A)^{\prime}$.
\end{lemma}
\begin{proof}
Let $f:(X,\mathcal{B}_{1},{\bf b}_{1})\rightarrow (Y,\mathcal{B}_{2},{\bf b}_{2})$ be given. Identify $X$ and $Y$ as the corresponding sets of point clusters given by the respective discrete extensions $\delta_{1}$ and $\delta_{2}$. Let $A\subseteq X$ be unbounded. Note that if $\sigma$ is an element of $\mathcal{U}X,$ then by the remarks after Theorem \ref{theorem111}, and by the bullet point (3) before Definition \ref{mapUf}, we have that
\[\sigma\in A^{\prime} \quad \iff \quad C{\bf b}_1A \text{ for all } C\in\sigma.\]
Also, recall that $\mathcal{U}f(\sigma)$ is given by
\[\mathcal{U}f(\sigma)=\{D\subseteq Y\mid D{\bf b}_{2}f(C) \text{ for all }C\in\sigma\}.\]
Let $\sigma \in A'.$ We wish to show that $\mathcal{U}f(\sigma)\in f(A)^{\prime}$. It will suffice to show that if $D\in\mathcal{U}f(\sigma)$ then $D{\bf b}_{2}f(A)$. Let $D\in\mathcal{U}f(\sigma)$. Then, by the definition of $\mathcal{U}f(\sigma)$, $D{\bf b}_{2}f(C) \text{ for all }C\in\sigma$. Since $A \in \sigma$, $D{\bf b}_{2}f(A)$. Thus, $\mathcal{U}f(\sigma)\in f(A)^{\prime}$, concluding the proof that $\mathcal{U}f(A^{\prime})\subseteq f(A)^{\prime}$.

Now assume that $f$ is a coarse proximity isomorphism with a coarse inverse $g.$ Let $\sigma' \in f(A)'.$ Then $D{\bf b}_2 f(A)$ for all $D \in \sigma'.$ We wish to show that $\sigma' \in \mathcal{U}f(A'),$ i.e., there exists $\sigma \in A'$ such that $\mathcal{U}f(\sigma)=\sigma'.$ Define
\[\sigma:=\mathcal{U}g(\sigma')=\{C\subseteq X\mid C{\bf b}_{1}g(D) \text{ for all }D\in\sigma'\}.\]
Let us show that $\sigma \in A'$ and $\mathcal{U}f(\sigma)=\sigma'.$ To see that $\sigma \in A',$ recall that $\sigma' \in f(A)'$ implies that $f(A) \in \sigma'.$ Consequently, $g(f(A)) \in \sigma.$ Since clusters in the boundary are closed under the $\phi$ relation (see Lemma $4.4$ in \cite{paper3}) and $g(f(A))\phi_1A$ by the definition of a coarse proximity isomorphism, we have that $A \in \sigma.$ To see that $\mathcal{U}f(\sigma)=\sigma',$ let $D \in \sigma'.$ To show that $D \in \mathcal{U}f(\sigma),$ we need to show that $D{\bf b}_{2}f(C) \text{ for all }C\in\sigma.$ Let $C \in \sigma.$ This means that $C{\bf b}_1 g(E)$ for all $E \in \sigma'.$ In particular, $C{\bf b}_1 g(D).$ Consequently,
\[f(C){\bf b}_2f(g(D)) \phi D.\]
Thus, $f(C){\bf b}_2 D,$ finishing the proof that $\sigma'\subseteq \mathcal{U}f(\sigma).$ But this implies that $\mathcal{U}f(\sigma)=\sigma'.$ Thus, $\mathcal{U}f(A^{\prime}) = f(A)^{\prime},$ as desired.
\end{proof}

\begin{theorem}
	Isomorphic coarse proximity spaces have the same asymptotic inductive dimension.
\end{theorem}
\begin{proof}
As could be expected, the proof will be by induction. Let $f:(X,\mathcal{B}_1, {\bf b}_1)\rightarrow (Y, \mathcal{B}_2, {\bf b}_2)$ be a coarse proximity isomorphism with coarse inverse $g$. If $asInd(X)=-1,$ then $X$ is bounded, which implies that $Y$ is bounded as well, giving $asInd(Y)=-1$. Now assume that the result holds up to (and including) $n-1$ and assume that $asInd(X)=n$. Let $A,B\subseteq Y$ be such that $A\bar{\bf b}_2B$. Then by Lemma \ref{coarse proximity is completely preserved by isomorphisms} we have that $g(A)\bar{\bf b}_1g(B)$. Because $asInd(X)=n$ we have that there is an asymptotic separator $C\subseteq X$ between $g(A)$ and $g(B)$ such that $asInd(C)\leq n-1$. By the definition of coarse proximity isomorphisms, we have that $fg(A)\phi A$ and $fg(B)\phi B$. Since $\phi$-related sets have the same trace (see Proposition 4.8 in \cite{paper3}), we have that $(fg(A))^{\prime}=A^{\prime}$ and $(fg(B))^{\prime}=B^{\prime}$. By inductive hypothesis, we have that $asInd(f(C))=asInd(C)\leq n-1$. It will then suffice to show that $f(C)$ is an asymptotic separator between $A$ and $B$. However, this follows from Lemma \ref{sending traces to traces}. Therefore, $asInd(Y)=n$. 
\end{proof}

Dranishnikov's question regarding the relation between $dim(\nu X)$ and $asInd(X)$ for proper metric spaces (where $dim(\nu X)$ denotes the covering dimension of Higson corona) can be generalized to:
	
\begin{center}

Does $dim(\mathcal{U}X)=asInd(X)$ for all coarse proximity spaces $(X,\mathcal{B},{\bf b})?$

\end{center}

The answer to this generalized question is negative. To see it, we need the following lemmas.

\begin{lemma}\label{realizing separation}
	Let $(X,\mathcal{B},{\bf b})$ be a coarse proximity space, and let $A_{1},B_{1}\subseteq\mathcal{U}X$ be disjoint closed subsets. Then there are unbounded subsets $A_2,B_2\subseteq X$ such that $A_2\bar{ {\bf b}}B_2$, $A_{1}\subseteq A_2^{\prime}$, and $B_{1}\subseteq B_2^{\prime}$.
\end{lemma}
\begin{proof}
Let $\delta$ be the discrete extension of ${\bf b}$, and $\mathfrak{X}$ the Smirnov compactification of $(X,\delta)$ with $X$ being identified with the set of point clusters of $\mathfrak{X}$. Let $A_{1},B_{1}\subseteq\mathcal{U}X$ be given. Because $\mathcal{U}X$ is a closed and compact subset of $\mathfrak{X},$ we have that $A_{1}$ and $B_{1}$ are disjoint closed subsets of $\mathfrak{X}$. Then using Urysohn's lemma gives a continuous (proximity) function $f:\mathfrak{X}\rightarrow[0,1]$ such that $f(A_{1})=0$ and $f(B_{1})=1$. We then define $A_2=f^{-1}([0,1/3])\cap X$ and $B_2=f^{-1}([2/3,1])\cap X$. These two sets are clearly disjoint. They are also nonempty and unbounded as $f^{-1}([0,1/3])$ and $f^{-1}([2/3,1])$  are compact neighborhoods of $A_{1}$ and $B_{1}$ in $\mathfrak{X},$ respectively. Now we will show that $A_{1}\subseteq A_2^{\prime}$. Let $\sigma\in A_{1}$. Let $U$ be an arbitrary open set in $\mathfrak{X}$ that contains $\sigma$. Then $f^{-1}([0,1/3))\cap U$ is an open set in $\mathfrak{X}$ that contains $\sigma.$ Because $X$ is dense in $\mathfrak{X}$, $f^{-1}([0,1/3))\cap U$ contains an element of $X$. Thus, $f^{-1}([0,1/3))\cap U \cap X \subset A_2 \cap U$ is not empty. Because $U$ was an arbitrary open set in $\mathfrak{X}$ containing $\sigma$, $\sigma$ is in the closure of $A_2$ in $\mathfrak{X}$, i.e., $\sigma \in A_2'.$ Showing that $B_{1}\subseteq B_2^{\prime}$ is similar.
\end{proof}

\begin{lemma}\label{subspace boundaries are subsets of the overboundary}
	Let $(X,\mathcal{B},{\bf b})$ be a coarse proximity space and $Y\subseteq X$ a nonempty subset. If $(Y,\mathcal{B}_{Y},{\bf b}_{Y})$ is the subspace coarse proximity structure, then the inclusion map $\iota:Y\rightarrow X$ induces an embedding $\mathcal{U}\iota:\mathcal{U}Y\rightarrow\mathcal{U}X$. Moreover, the image of $\mathcal{U}\iota$ is precisely $Y^{\prime}$.    
\end{lemma}
\begin{proof}
If $Y$ is bounded, then the corollary is trivially true, so let us assume that $Y$ is unbounded. Let $\delta_{Y}$ and $\delta$ be the discrete extensions of ${\bf b}_{Y}$ and ${\bf b}$, respectively. The inclusion map $\iota:Y\rightarrow X$ is trivially a coarse proximity map, and thus by remarks after Definition \ref{mapUf} gives a continuous map $\mathcal{U}\iota:\mathcal{U}Y\rightarrow\mathcal{U}X$. To show that this map is also injective, let $\sigma_{1},\sigma_{2}\in\mathcal{U}Y$ be clusters with identical image. Recall from Definition \ref{mapUf} that this image is given by
\[\{A\subseteq X\mid\forall B\in\sigma_{1}, A\delta\iota(B)\}=\mathcal{U}\iota(\sigma_{1})=\mathcal{U}\iota(\sigma_{2})=\{A\subseteq X\mid\forall B\in\sigma_{2}, A\delta\iota(B)\}.\]
Since $\iota(B)=B$ for all $B\subseteq Y$, we have that for all $B_{1}\in\sigma_{1}$ and $B_{2}\in\sigma_{2},$ $B_{1},B_{2}\in\mathcal{U}\iota(\sigma_{1})$. Hence, $B_{1}\delta B_{2},$ and subsequently $B_{1}\delta_{Y}B_{2}.$ Thus, if $B_{1}\in\sigma_{1}$, then $B_{1}\delta_{Y}C$ for all $C\in\sigma_{2}$. Therefore $B_{1}\in\sigma_{2}$. Similarly one show that if $B_{2}\in\sigma_{2},$ then $B_{2}\in\sigma_{1}$. Therefore, $\sigma_{1}=\sigma_{2}$ from which we conclude that $\mathcal{U}\iota$ is an embedding. 

To see that the image of $\mathcal{U}\iota$ is $Y^{\prime},$ let $\mathcal{U}\iota(\sigma)$ be an arbitrary element of the image of $\mathcal{U}\iota$ for some $\sigma \in \mathcal{U}Y.$ 
To see that $\mathcal{U}\iota(\sigma) \in cl_{\mathfrak{X}}(Y),$ notice that for all $B \in \sigma,$ we have that $Y{ \bf b} \iota (B)$ (it is because $\iota(B)=B$ and $B \cap Y$ is unbounded, since $B$ is unbounded and $B \subseteq Y$). Thus, $Y \in \mathcal{U}\iota(\sigma),$ i.e., $\mathcal{U}\iota(\sigma) \in cl_{\mathfrak{X}}(Y).$ Consequently, $\mathcal{U}\iota(\sigma) \in Y'.$ Conversely, let $\sigma \in Y'.$ In particular, $Y \in \sigma.$ Then by Theorem $5.16$ in \cite{proximityspaces}, there exists a unique cluster $\sigma'$ in $(Y, \delta_Y)$ contained in $\sigma,$ namely
\[\sigma':=\{A \subseteq Y \mid A \in \sigma\}.\]
Notice that $\sigma' \in \mathcal{U}Y$ (otherwise $\sigma$ would contain a bounded set, since bounded sets in $Y$ are bounded in $X$).
Consider $\mathcal{U}\iota(\sigma').$ It is a cluster in $\mathcal{U}X.$ We will show that  $\sigma=\mathcal{U}\iota(\sigma'),$ which will finish the proof. Let $A \in  \sigma.$ Notice that any element of $\sigma'$ is also an element of $\sigma$. Consequently, for all $B \in \sigma'$ we have that $A\delta B.$ In other words, $A\delta \iota(B)$ for all $B \in \sigma',$ i.e., $A \in \mathcal{U}\iota(\sigma').$ Thus, $\sigma \subseteq \mathcal{U}\iota(\sigma'),$ which shows that $\sigma = \mathcal{U}\iota(\sigma'),$ as desired.
\end{proof}

In light of the above lemma, we may identify the boundary $\mathcal{U}Y$ of a subspace $Y$ of a coarse proximity space $(X,\mathcal{B},{\bf b})$ with its trace $Y^{\prime}$.

\begin{theorem}\label{asymptotic inductive is larger than large inductive of corona}
	$Ind(\mathcal{U}X)\leq asInd(X)$ for all coarse proximity spaces $(X,\mathcal{B},{\bf b})$.
\end{theorem}
\begin{proof}
	The proof is by induction on $asInd(X)$. If $asInd(X)=-1,$ then $X$ is bounded and $\mathcal{U}X=\emptyset$ which implies that $Ind(\mathcal{U}X)=-1$. Now assume that the result holds for $asInd(X)<n$. Assume that $asInd(X)=n$ and let $A_{1},B_{1}\subseteq\mathcal{U}X$ be disjoint closed subsets. Then $A_{1}$ and $B_{1}$ are disjoint closed (compact) subsets in $\mathfrak{X}$, the Smirnov compactification of the proximity space $(X,\delta)$, where $\delta$ is the discrete extension of ${\bf b}$. By Lemma \ref{realizing separation}, there are unbounded subsets $A_{2},B_{2}\subseteq X$ such that $A_{2}\bar{ {\bf b}}B_{2}$, $A_{1}\subseteq A_{2}^{\prime}$, and $B_{1}\subseteq B_{2}^{\prime}$. Because $asInd(X)=n$, there is an asymptotic separator $C\subseteq X$ between $A_{2}$ and $B_{2}$ such that $asInd(C)\leq n-1$. Because $A_{1}\subseteq A_{2}^{\prime}$ and $B_{1}\subseteq B_{2}^{\prime}$ we have that $C^{\prime}$ is a topological separator between $A_{1}$ and $B_{2}$. Since by Lemma \ref{subspace boundaries are subsets of the overboundary} we have that $\mathcal{U}C=C^{\prime}$, by inductive hypothesis we get
	\[Ind(C^{\prime})=Ind(\mathcal{U}C)\leq asInd(C)\leq n-1.\qedhere\]
\end{proof}

Since covering dimension is always smaller or equal to large inductive dimension on compact Hausdorff spaces, we get the following corollary.

\begin{corollary}
$dim(\mathcal{U}X)\leq asInd(X)$ for all coarse proximity spaces $(X,\mathcal{B},{\bf b})$.\qed
\end{corollary}

In 1958, P. Vopenka described a class of compact Hausdorff spaces for which the large inductive dimension is strictly greater than the covering dimension (see \cite{nagami}). Since every compact Hausdorff space can be realized as the boundary of a coarse proximity space (see Corollary 5.3 in \cite{paper3}), Theorem \ref{asymptotic inductive is larger than large inductive of corona} answers the question of Dranishnikov generalized to coarse proximity spaces, as stated in the following corollary.

\begin{corollary}
	There is a coarse proximity space $X$ for which $dim(\mathcal{U}X)<asInd(X)$. \qed
\end{corollary}

Since we know that $Ind(\mathcal{U}X)\leq asInd(X)$ for all coarse proximity spaces $(X,\mathcal{B},{\bf b}),$ the next most natural question is under what conditions do $Ind(\mathcal{U}X)$ and $asInd(X)$ coincide? We provide answers to this question in the next section.

\section{Relationship between $asInd(X)$ and $Ind(\mathcal{U}X)$}\label{maintheorems}

In the previous section, the proof of Theorem \ref{asymptotic inductive is larger than large inductive of corona} suggested that the gap (or lack thereof) between $Ind(\mathcal{U}X)$ and $asInd(X)$ for a coarse proximity space $X$ is tied up with which closed sets $K\subseteq\mathcal{U}X$ appear as the traces of unbounded subsets of $X$. It is an easy exercise to show that if $X$ is an unbounded proper metric space and $x$ is an element of the Higson corona $\nu X$, then there is no unbounded set whose trace is precisely $\{x\}$. Being unable to detect all closed subsets of the Higson corona in this way makes closing the gap between $asInd(X)$ (when $X$ is equipped with its metric coarse proximity structure) and $Ind(\nu X)$ by simply modifying the proof of Theorem \ref{asymptotic inductive is larger than large inductive of corona} impossible. In general, whether or not $asInd(X)=Ind(\nu X)$ is an open question. However, in this section we will describe scenarios in which $asInd(X)=Ind(\mathcal{U}X)$ for certain classes of coarse proximity spaces. 
\vspace{\baselineskip}

The class of spaces in which $asInd(X)=Ind(X)$ is the obvious one suggested by the proof of Theorem \ref{asymptotic inductive is larger than large inductive of corona}. Specifically, this is the class of spaces for which every closed subset of the boundary can be realized as the trace of an unbounded set. 

\begin{definition}
	Let $(X,\mathcal{B},{\bf b})$ be a coarse proximity space with boundary $\mathcal{U}X$. A closed subset $C\subseteq\mathcal{U}X$ is called {\bf traceable} if there is some unbounded $A\subseteq X$ such that $A^{\prime}=C$. We say that $\mathcal{U}X$ is {\bf completely traceable} if every nonempty closed $C\subseteq\mathcal{U}X$ is traceable.
\end{definition} 

\begin{theorem}
	If $(X,\mathcal{B},{\bf b})$ is a coarse proximity space whose boundary is completely traceable, then $asInd(X)=Ind(\mathcal{U}X)$.
\end{theorem}
\begin{proof}
	In light of Theorem \ref{asymptotic inductive is larger than large inductive of corona}, it will suffice to show that $asInd(X)\leq Ind(\mathcal{U}X).$ The proof will be by induction on $Ind(\mathcal{U}X)$. If $Ind(\mathcal{U}X)=-1,$ then $\mathcal{U}X=\emptyset,$ which implies that $X$ is bounded, and thus $asInd(X)=-1$. Assume then that the result holds for $Ind(\mathcal{U}X)<n$ and assume that $Ind(\mathcal{U}X)=n$. Let $A,B\subseteq X$ be such that $A\bar{ {\bf b}}B$. We may assume without loss of generality that both $A$ and $B$ are unbounded. Because $A\bar{ {\bf b}}B$, we have that $A^{\prime}$ and $B^{\prime}$ are disjoint closed subsets of $\mathcal{U}X$. Because $Ind(\mathcal{U}X)=n$, there is a closed separator $K\subseteq\mathcal{U}X$ between $A^{\prime}$ and $B^{\prime}$ such that $Ind(K)\leq n-1$. Because $\mathcal{U}X$ is completely traceable, there is an unbounded set $D\subseteq X$ such that $D^{\prime}=K$. Equipping $D$ with its subspace coarse proximity structure, we have that $\mathcal{U}D$ is homeomorphic to $K$, and therefore by the inductive hypothesis we have that
	 \[asInd(D)\leq Ind(\mathcal{U}D)=Ind(K)\leq n-1.\]
Therefore, $asInd(X)\leq n=Ind(\mathcal{U}X)$, yielding the desired result.
\end{proof}

Which spaces have completely traceable boundaries? One such class is given by spaces whose boundaries in their compactifications are $Z$-sets.

\begin{definition}
Let $X$ be a topological space and $A$ a closed subset of $X.$ Then $A$ is called a \textbf{$Z$-set} if there exists a homotopy $H: X \times [0,1] \to X$ such that $H(X,0)=id_X$ and $H(X, t) \subseteq (X \setminus A)$ for all $t \in (0,1].$
\end{definition}

Many spaces have boundaries that are $Z$-sets. For example, Gromov boundary of a hyperbolic proper metric space and the visual boundary of a proper CAT$(0)$ space are $Z$-sets in their respective compactifications (see for example \cite{Z_sets}).

\begin{theorem}
	Let $X$ be a locally compact Hausdorff space and $\overline{X}$ a compactification of $X$. Let $(X,\mathcal{B},{\bf b})$ be the coarse proximity structure on $X$ induced by the compactification $\overline{X},$ i.e.,  $\mathcal{B}$ is the collection of all sets whose closures in $X$ are compact, and ${\bf b}$ is defined by
	\[A{\bf b}B\iff cl_{\overline{X}}(A)\cap cl_{\overline{X}}(B)\cap(\overline{X}\setminus X)\ne\emptyset.\]
	If $\overline{X}\setminus X$ is a $Z$-set in $\overline{X},$ then $\mathcal{U}X$ (identified with $\overline{X}\setminus X$) is completely traceable.
\end{theorem}
\begin{proof}
	Denote the Smirnov compactification of $X$ given by the discrete extension of ${\bf b}$ by $\mathfrak{X}$. Then $\mathfrak{X}=X\cup\mathcal{U}X$. Let $H:\mathfrak{X}\times I\rightarrow\mathfrak{X}$ be a homotopy that witnesses $\mathcal{U}X$ being a $Z$-set of $\mathfrak{X}$, and let $K\subseteq \mathcal{U}X$ be a nonempty closed subset. Define	
	\[K_{*}=H(K,[0,1])\text{ and }D=K_{*}\setminus K.\]
	Then $K\subseteq K_{*}$ and $D=H(K,(0,1])$ is an unbounded subset of $X$ such that $K\subseteq D^{\prime}$. This latter statement can be seen as given $x\in K,$ we have that given any sequence $(s_{n})$ in $(0,1]$ converging to $0,$ we have that the sequence $(H(x,s_{n}))$ converges to $x$. To see that $D^{\prime}$ is precisely $K,$ we simply note that by the closed map lemma $K_{*}=K\cup D$ is closed in $\mathfrak{X},$ and thus
\[D'=cl_{\mathfrak{X}}(D) \cap \mathcal{U}X\subseteq cl_{\mathfrak{X}}(K_{*}) \cap \mathcal{U}X=K_{*}\cap \mathcal{U}X=K. \qedhere\] 
\end{proof}

\begin{corollary}\label{corollary111}
	Let $X$ be a locally compact Hausdorff space and $\overline{X}$ a compactification of $X$ such that $\overline{X}\setminus X$ is a $Z$-set in $\overline{X}.$ Let $(X,\mathcal{B},{\bf b})$ be the coarse proximity structure on $X$ induced by the compactification $\overline{X}$. Then
\[asInd(X)=Ind(\mathcal{U}X)=Ind(\overline{X}\setminus X). \eqno\qed\]
\end{corollary}

\begin{corollary}
Let $X$ be a proper $\delta$-hyperbolic metric space, $\overline{X}$ its Gromov compactification, and $\overline{X}\setminus X$ its Gromov boundary. Then,
\[asInd(X)=Ind(\overline{X}\setminus X).\]
\end{corollary}
\begin{proof}
This follows from Corollary \ref{corollary111} and the fact that $\overline{X}\setminus X$ is a $Z$-set in $\overline{X}$  (see for example \cite{Z_sets}).
\end{proof}

Another class of spaces with completely traceable boundaries are spaces  admitting metrizable compactifications. Such compactifications were described in detail using controlled products in \cite{yamauchi}. 

\begin{theorem}\label{metrizeablecompeletelytraceable}
	Let $X$ be a locally compact Hausdorff space and $\overline{X}$ a metrizable compactification of $X$. Let $(X,\mathcal{B},{\bf b})$ be the coarse proximity structure on $X$ induced by the compactification $\overline{X}$. Then $\overline{X}\setminus X$, identified with $\mathcal{U}X$, is completely traceable.
\end{theorem}
\begin{proof}
Let $K\subseteq\mathcal{U}X$ be a given nonempty closed subset, and let $d$ be a metric on $\overline{X}$ that is compatible with the topology on $\overline{X}$. For each $n\in\mathbb{N},$ let 
\[\mathcal{C}_{n}=\{B(x_{n_1},1/n), B(x_{n_2},1/n)\ldots,B(x_{n_m},1/n)\}\]
be a finite open cover of $K$ where each $x_{n_{i}}\in K$. As $X$ is dense in $\overline{X},$ we have that $B(x_{n_{i}},1/n)$ intersects $X$ nontrivially for each $i$ and $n$. We then let $y_{n_{i}}$ be an element of $B(x_{n_{i}},1/n)\cap X$ for each $i$ and $n$. Define $D\subseteq X$ to be the collection of all these $y_{n_{i}}.$ 
We claim that $D^{\prime}=K$. To see that $D^{\prime}\subseteq K,$ let $(y_{n})$ be an unbounded sequence in $D$ that converges to some element of $\overline{X}\setminus X$. Because this sequence is unbounded, we have that for each $n\in\mathbb{N}$ there is some $m\geq n$ such that $y_{m}$ is in an element of $\mathcal{C}_{n}$. Then there is a subsequence $(y_{n_{k}})$ of $(y_{n})$ such that $\lim_{k\to\infty}d(y_{n_{k}},K)=0$, which implies that $(y_{n_{k}})$ converges to a point in $K$. As each subsequence of $(y_{n})$ must converge to the same point as $(y_{n}),$ we have that $(y_{n})$ converges to a point of $K$, which gives us that $D^{\prime}\subseteq K$. Now let $x\in K$. For each $n\in\mathbb{N},$ there is some $x_{n_{i}}\in K$ such that $x\in B(x_{n_{i}},1/n)$. Choosing one such open ball for each $n\in\mathbb{N},$ we specify an unbounded sequence in $D$ that converges to $x$. Therefore, $K\subseteq D^{\prime}$ and consequently $K=D^{\prime}$. Thus, $\mathcal{U}X$ is completely traceable.
\end{proof}

\begin{corollary}\label{metrizablecompactificationdimension}
	Let $X$ be a locally compact Hausdorff space and $\overline{X}$ a metrizable compactification of $X$. Let $(X,\mathcal{B},{\bf b})$ be the coarse proximity structure on $X$ induced by the compactification $\overline{X}$. Then 
\[asInd(X)=Ind(\mathcal{U}X)=Ind(\overline{X}\setminus X). \eqno\qed\]
\end{corollary}

\section{Freudenthal boundary}\label{Freud}

In this section we will apply the results of the preceding section to the Freudenthal boundary of a class of metrizable spaces. We begin by recalling the following definition.

\begin{definition}
Let $X$ be a topological space and $A$ and $B$ two disjoint subsets of $X$. We say that $A$ and $B$ are {\bf separated by a compact set} in $X$ if there is a compact set $K\subseteq X$ that separates $A$ and $B$ in $X$. That is, there is a compact set $K\subseteq X$ that is disjoint from $A$ and $B$ and is such that $X\setminus K$ is the union of two disjoint open sets $U_{1},U_{2}$ with $A\subseteq U_{1}$ and $B\subseteq U_{2}$.

\end{definition}

The Freudenthal compactification is then, as described in \cite{isbell}, given by the following characterization.

\begin{definition}
Let $X$ be a locally compact Hausdorff space. The {\bf Freudenthal compactification} of $X$ is the Smirnov compactification of $X$ corresponding to the separated proximity $\delta_{F}$ on $X$ (hereafter referred to as the Freudenthal proximity) defined by

\[A\bar{\delta}_{F}B\iff A\text{ and }B\text{ are separated in }X\text{ by a compact set.}\]

\noindent
for subsets $A,B\subseteq X$. Denoting this compactification by $FX$, the boundary $FX \setminus X$ is called the {\bf Freudenthal boundary} of $X$.

\end{definition}

In \cite{paper3}, the coarse proximity structure on a locally compact Hausdorff space $X$ whose boundary is homeomorphic to the Freudenthal boundary of $X$ was described. Specifically, if $X$ is a locally compact Hausdorff space, then let $\mathcal{B}_{C}$ be the collection of all precompact subsets of $X$ (i.e., subsets whose closure is compact). Then, $\mathcal{B}_{C}$ is a bornology and the binary relation ${\bf b}_{F}$ on the powerset of $X$ defined by

\[A{\bf b}_{F}B\iff (A\setminus D)\delta_{F}(B\setminus D)\text{ for all }D\in\mathcal{B}_{C}\]

\noindent
is a coarse proximity relation. The coarse proximity structure $(X,\mathcal{B}_{C},{\bf b}_{F})$ is called the {\bf Freudenthal coarse proximity structure}. Now, recall that for a topological space $X$ and a point $x\in X$, the {\bf quasi-component} of $X$ is the intersection of all open and closed (clopen) subsets of $X$ that contain $x$. Distinct quasi components in $X$ are disjoint and closed. The {\bf quasi-component} space of $X$, denoted $Q(X)$, is the set of quasi-components in $X$ with the topology generated by the clopen subsets of $X$. The following theorem appears as a corollary in \cite{dimensionandextensions}:

\begin{theorem}\label{metrizablefreudenthal}
Let $X$ be a locally compact, separable metrizable space. Then $FX$ is metrizable if and only if $Q(X)$ is compact.
\end{theorem}

\noindent
In light of this theorem and Theorem \ref{metrizeablecompeletelytraceable}, we get the following corollary:

\begin{corollary}
Let $X$ be a locally compact, separable, metrizable space with compact quasicomponent space. Equip $X$ with the Freudenthal coarse proximity structure $(X, {\bf b}_{F},\mathcal{B}_{C})$. Then $\mathcal{U}X$, which is homeomorphic to the Freudenthal boundary, is completely traceable in the Freudenthal compactification of $X$.
\end{corollary}

\noindent
Combining this with Corollary \ref{metrizablecompactificationdimension}, we get the following:

\begin{corollary}
Let $X$ be a locally compact, separable, metrizable space with compact quasicomponent space. If $X$ is equipped with the Freudenthal coarse proximity structure, then

\[asInd(X)=Ind(\mathcal{U}X)=Ind(FX\setminus X)\]

\noindent
where $FX$ is the Freudenthal compactification of $X$. 
\end{corollary}

A more in depth discussion of the Freudenthal compactification, its boundary, and its application to groups in the context of coarse geometry can be found in \cite{coarsefreudenthal}.

\bibliographystyle{abbrv}
\bibliography{Asymptotic_Inductive_Dimension_of_Coarse_Proximity_Spaces}{}

\end{document}